\documentclass{article}

\usepackage{fancyhdr} 
\usepackage{ifthen}
\newboolean{first}
\setboolean{first}{true}
\usepackage{amsthm}
\usepackage[all,cmtip]{xy}
\usepackage{amsmath,amsfonts}
\usepackage{xcolor}
\usepackage{graphicx}
\usepackage{mathabx,mathrsfs}
\usepackage{extarrows}
\usepackage{cite}

\title{\LARGE \(Aut(\Gamma)=\det(\mathcal{M}_\Gamma)\)}
\author{Zhonggan Huang \\11510296@mail.sustc.edu.cn}

\begin{document}
\maketitle

\begin{abstract}
In this article, we will show that the automorphism group of any hypergraph is essentially equal to the determinant of some matrix over a ring generated from the set of ground points. With this, we are also able to determine whether two graphs are isomorphic.
\end{abstract}
\newtheorem{defn}{Definition}[section]
\newtheorem{thm}{Theorem}[section]
\newtheorem{prop}{Proposition}[section]
\newtheorem{rmk}[thm]{Remark}
\newtheorem{exmp}[defn]{Example}

\section{Introduction}
For a long time people have been investigating the algebraic structures of graphs, one of which is their automorphism groups~\cite{n1}. Many people think it quite hard to obtain the automorphism group of an arbitrary graph, and even regard this problem as a special class in the computation complexity hierachy. But now, in this article, we've found a way to better understand how its automorphism group forms and transforms. We at first use the gound set to generate a ring called \emph{the Ring of Partials}, which was recognized as the power set of the symmetric group \(S_X\). Indeed, they share the same set and the same structure, but we give the elements different names for simplicity of calculations. We observe that in this ring, the behaviors of the gound points are in a highly regular and comfortable way. In fact, we find that the automorphism group is exactly the determinant of some matrix over the ring, which means that \emph{Graph Automorphism Problem} is not as hard as we thought before. This striking fact also applies to the \emph{Graph Isomorphism Problem}.

\section{A Baby Idea}
Given a hypergraph \(\Gamma\), we let \(n=|\Gamma|\), and write \(\Gamma=\left\{A_1,\,A_2,\,\dots,\,A_n\right\}\). We define
\(Sgl(\Gamma)\coloneq \bigcup \Gamma = \left\{x\in X\,;\, x\in A_i,\;
\text{for some}\;i \right\}\). We call \(x\in X\) \(\Gamma\)-\emph{singular} if \(x\in Sgl(\Gamma)\), otherwise we call it \(\Gamma\)-\emph{free}. A hypergraph is called \emph{spanning} if its singular set is exactly \(X\). For the rest of the article, we assume all the hypergraphs we talk about are spanning.

We call a point \(x\in X\) of \emph{degree} \(k\) in \(\Gamma\) if \(x\) is included in exactly \(k\) elements in \(\Gamma\). And we define \(k\)-\emph{class} of \(\Gamma\) as \(\mathscr{C}_\Gamma^k=\{x\in X\,;\, \text{deg } x =k\}\). Thus, the set of singular points is exactly the subset of all points of degree bigger than 0.

For the hypergraph \((X,\Gamma)\), of which we can define the \(k\)-\emph{section}
\[
\Gamma_k\coloneq \{A\in\Gamma\,;\,|A|=k\}\,,\, 1\le k\le m
\]
then \((X,\Gamma_k)\) becomes the maximal \(k\)-homogeneous sub-hypergraph of \((X,\Gamma)\). It is quite clear that a permutation will not bring an edge to a different section, and so we can show

\begin{prop}
	\[
	Aut(\Gamma)=\bigcap_{k=1}^{m} Aut(\Gamma_k)
	\]\label{sectionintersection}
\end{prop}
\begin{proof}
	An automorphism of \(\Gamma\) will not bring an edge to a different section, then it will preserve edges in each section; a permutation that preserves each section of \(\Gamma\) will certainly preserve \(\Gamma\). \label{sectionpreserve}
\end{proof}

This easy result is not our goal, but everything becomes much more interesting when we observe the following facts. Now we will focus on the automorphisms of \(k\)-homogeneous hypergraphs, and later we will extend the result.

\begin{defn}
	\begin{itemize}
		\item[\(a.\)] To each \(A_i\in\Gamma_k\), we denote its stabilizer \(G_i=S_{A_i}\otimes S_{X\backslash A_i}\);	
		\item[\(b.\)] Given \(A_i,A_j\in \Gamma_k\), there should be a permutation that brings \(A_i\) to \(A_j\), and we denote that permutation \(\sigma_{ij}\). (We will give a general way to construct \(\sigma_{ij}\) soon.) 
	\end{itemize}
\end{defn}

\begin{prop}
	Given two permutations \(\sigma_{ij}\) and \(\sigma_{ij}^\prime\), we have
	\[
	\sigma_{ij}G_i=\sigma_{ij}^\prime G_i=G_j\sigma_{ij}=G_j\sigma_{ij}^\prime
	\]
	\label{transfercoset}
\end{prop}
\begin{proof}
	Observe that \(\sigma_{ij}^{-1}\sigma_{ij}^\prime\) will bring \(A_i\) to itself, and so \(\sigma_{ij}^{-1}\sigma_{ij}^\prime \in G_i\); Similarly we have \(\sigma_{ij}^\prime\sigma_{ij}^{-1}\) brings \(A_j\) to itself, and so \(\sigma_{ij}^\prime\sigma_{ij}^{-1} \in G_j\). To see \(\sigma_{ij}G_i=G_j\sigma_{ij}\), we have to observe that for \(g_j\in G_j\), \(g_j\mapsto \sigma_{ij}^{-1}g_j\sigma_{ij}\) defines an injection from \(G_j\) to \(G_i\); since both are of the same finite cardinality, we have that the mapping is onto.
\end{proof}

\begin{thm}
	Let \(\sigma_{ij}\in S_{X}\) such that it takes \(A_i\) to \(A_j\), then
	\[Aut(\Gamma_k)=\bigcap_{i}\bigcup_j \sigma_{ij} G_i\]
	\label{interunion}
\end{thm}
\begin{proof}
	Given \(g\in Aut(\Gamma_k)\), then for any \(A_i\in\Gamma_k\), \(g(A_i)\in \Gamma_k\). That is to say, there will be an \(A_j\in \Gamma_k\) such that \(g(A_i)=A_j\). Thus, for any \(1\le i\le n\), we have that 
	\[
	g\in \bigcup_{j=1}^{n} \sigma_{ij}G_i
	\]
	where each \(\sigma_{ij}G_i\) collects all the permutations that bring \(A_i\) to \(A_j\). Thus \(\displaystyle Aut(\Gamma_k)\subseteq \bigcap_{i=1}^{n}\bigcup_{j=1}^{n} \sigma_{ij}G_i\). 
	
	Conversely, for \(\displaystyle g\in\bigcap_{i=1}^{n}\bigcup_{j=1}^{n} \sigma_{ij}G_i\), we see that \(g\) will only bring elements in \(\Gamma_k\) to those of itself.
\end{proof}

\begin{rmk}
	 We would like to define \(\sigma_{ij}\)'s as follows:
	 From \(A_i\) to \(A_j\), we consider \(\sigma_{ij}\) a permutation that maps \(\Delta_{ij}\coloneq A_i\backslash A_j\) to \(\Delta_{ji}\coloneq A_j\backslash A_i\) and an identity elsewhere. More precisely, we first order \(\Delta_{ij}\) and \(\Delta_{ji}\) respectively, and then take elements in \(\Delta_{ij}\) to those of the same order in \(\Delta_{ji}\). If we define \(\sigma_{ij}\)'s in this way, we see that they will thus be involutions. As an example, for \(A_1=\{1,2,3,4,7,8\}\) and \(A_2=\{3,4,1,5,34,17\}\), we define \(\sigma_{12}\coloneq(25)(7,17)(8,34)\). \label{constructionofsigma}
\end{rmk}

To see the essential structure of the automorphism group of a given \(k\)-homogeneous hypergraph \(\Gamma=\{A_1,\,\dots,\,A_n\}\), we first take a look at a natural subgroup---the collection of all permutations that preserve all the edges at the same time. Notice that this group, denoted by \(\bigcap_{i=1}^{n}G_i\), is normal in \(Aut(\Gamma)\), because for \(g\in Aut(\Gamma),\;\sigma\in \bigcap_{i=1}^{n}G_i\), we still have \(g\sigma g^{-1}\in \bigcap_{i=1}^{n}G_i\). We will denote this subgroup \(\mathcal{K}_\Gamma\).

According to \emph{Theorem} \ref{interunion}, we can do further calculations to analyse the automorphism groups in general. To begin with, we write the cosets into a matrix:
\[
\mathcal{M}_\Gamma\coloneq
\left(\begin{matrix}
\sigma_{11}G_1 & \sigma_{12}G_1 & \cdots&\sigma_{1n}G_1\\
\sigma_{21}G_2 & \sigma_{22}G_2 & \cdots&\sigma_{2n}G_2\\
\vdots&  \vdots&  \ddots&\vdots\\
\sigma_{n1}G_n & \sigma_{n2}G_n & \cdots &\sigma_{nn}G_n
\end{matrix}
\right)
\]

According to \emph{Proposition} \ref{transfercoset}, we see the above matrix can be reformulated as:

\[
\mathcal{M}_\Gamma=
\left(\begin{matrix}
G_1 & G_2\sigma_{12} & \cdots&G_n\sigma_{1n}\\
G_1\sigma_{21} & G_2 & \cdots&G_n\sigma_{2n}\\
\vdots&\vdots&\ddots&\vdots\\
G_1\sigma_{n1} & G_2\sigma_{n2}& \cdots &G_n
\end{matrix}
\right)
\]

Observe any two entries that are in the same column, say the \(i\)-th, are distinct cosets of \(G_i\), and hence by group theory, the two must be disjoint.

Notice in \emph{Theorem} \ref{interunion}, what we did is just take the union of each row of \(\mathcal{M}\), and then take the intersection of the unions. Basic calculations tell us that it is equivalent to first pick out exactly one entry from each row, and take their intersections, and then take the union of the intersections. From prior analysis, those intersections, in which there are more than two entries in the same column, must vanish. That is to say, an intersection
\[
\sigma_{1j_1}G_1\bigcap\sigma_{2j_2}G_2\bigcap\cdots\bigcap\sigma_{nj_n}G_n
\]
must vanish, when
\[
\left(\begin{matrix}
1&2&\cdots&n\\
j_1&j_2&\cdots&j_n
\end{matrix}
\right)
\]
is not a permutation. Therefore, we have the following proposition.

\begin{prop}
	\[\bigcap_{i}\bigcup_j \sigma_{ij} G_i=\bigcup_{j\_\in S_n} \bigcap_{i=1}^{n}\sigma_{ij_i}G_i\]\label{determinantuse}
\end{prop}

From the above proposition, one will observe that, the number of the intersections will be no more than the number of permutations of \([n]\), \emph{i.e.} \(n!\).

Recall that \(\mathcal{K}_\Gamma\) is a normal subgroup of \(Aut(\Gamma)\), and if we suppose that \(l=|Aut(\Gamma)|/|\mathcal{K}_\Gamma|\), then there will be \(g_1=(1),\,g_2,\,\dots,,\,g_l\) in \(Aut(\Gamma)\) such that
\[
Aut(\Gamma)=g_1\mathcal{K}_\Gamma \bigsqcup g_2\mathcal{K}_\Gamma\bigsqcup\cdots\bigsqcup g_l\mathcal{K}_\Gamma
\]

\begin{prop}
	For \(j\_\) a permutation of edges, we have that \(\bigcap_{i=1}^{n}\sigma_{ij_i}G_i\) will not vanish if and only if it is a coset of \(\mathcal{K}_\Gamma\) in \(Aut(\Gamma)\).\label{cosetiff}
\end{prop}

\begin{proof}
	Suppose for every \(g\in Aut(\Gamma)\), \(g\mathcal{K}_\Gamma\ne\bigcap_{i=1}^{n}\sigma_{ij_i}G_i\). Recall left multiplication by \(g\) induces a bijection over \(Aut(\Gamma)\), there is 
	\[
	g\mathcal{K}_\Gamma=g\left(\bigcap_{i=1}^{n}G_i\right)
	=\bigcap_{i=1}^{n}gG_i
	\]
	Then, there must be some \(i_0\) such that \(gG_{i_0}\ne \sigma_{i_0j_{i_0}}G_{i_0} \), and since they are both cosets of \(G_{i_0}\), they must be disjoint, and thus the assumption can be reformulated as:
	\[
	\forall g\in Aut(\Gamma),\, g\mathcal{K}_\Gamma\text{ is disjoint from }\bigcap_{i=1}^{n}\sigma_{ij_i}G_i
	\]
	Thus, \(\bigcap_{i=1}^{n}\sigma_{ij_i}G_i\) is disjoint from \(Aut(\Gamma)\), which implies that \(\bigcap_{i=1}^{n}\sigma_{ij_i}G_i\) is empty. The reverse is evident.
\end{proof}

\vspace*{5pt}

Given an element in \(\bigcap_{i=1}^{n}\sigma_{ij_i}G_i\), we shall see it will induce a permutation 
\[\left(\begin{matrix}
1&2&\cdots&n\\
j_1&j_2&\cdots&j_n
\end{matrix}
\right)\]
over the edges \(\Gamma=\{A_1,\,\dots,\,A_n\}\) by permuting the indices. Then, we have a corollary.
\begin{thm}
	\(Aut(\Gamma)/\mathcal{K}_\Gamma\) can be naturally embeded into \(\mathcal{S}_n\), and \(j\in \mathcal{S}_n\) is in \(Aut(\Gamma)/\mathcal{K}_\Gamma\) if and only if \(\bigcap_{i=1}^{n}\sigma_{ij_i}G_i\) is not empty.
\end{thm}

\begin{proof}
	For the first part, let 
	\[
	\begin{split}
	\psi: Aut(\Gamma)&\longrightarrow \mathcal{S}_n\\
	g&\longmapsto\psi(g)
	\end{split}
	\]
	here \(A_{\psi_g(i)}\coloneq gA_i\), for any \(A_i\in\Gamma,\,\forall 1\le i\le n\). It is easy to verify that the map is a well-defined homomorphism. If \(\psi(f)=id\in \mathcal{S}_n\) for some \(f\), then we see \(fA=A\) for all \(A\) in \(\Gamma\), which implies that \(f\in \mathcal{K}_\Gamma\). It is then clear to see that the kernel of \(\psi\) is exactly \(\mathcal{K}_\Gamma\). By fundamental theorem of homomorphisms, one should find the result. 
	
	For the second part, for \(j\in\mathcal{S}_n\), we may write it as:
	\[
	\left(\begin{matrix}
	1&2&\cdots&n\\
	j_1&j_2&\cdots&j_n
	\end{matrix}
	\right)
	\]
	then if \(j\) is in \(Aut(\Gamma)/\mathcal{K}_\Gamma=\text{Im}(\psi)\), there must be an element \(v\) in \(Aut(\Gamma)\) such that \(\psi(v)=j\). Notice \(v\in \sigma_{i{j_i}}G_i\) for all \(1\le i\le n\). And hence \(v\mathcal{K}_\Gamma\) is not empty. The reverse is trivial.
\end{proof}

It is reasonable to think of \(\mathcal{K}_\Gamma\)'s as plain groups. To illustrate this fact, we give a classification of \(\mathcal{K}_\Gamma\), where \(\Gamma\)'s are assumed to be spanning 2-homogeneous hypergraphs, \emph{i.e.} simply spanning graphs. We embed the symmetric group of a subset into \(S_X\) by making its elements identity on its complement, and call it the \emph{symmetry} of the subset.

\begin{defn}[\emph{Radical}]
	In a graph, we call an edge a \textbf{radical} if its vertices merely lie in it.
\end{defn}

\begin{prop}
	A spanning graph \(\Gamma\) have trivial \(\mathcal{K}_\Gamma\) if and only if it has no radicals. Moreover, \(\mathcal{K}_\Gamma\) will be the direct product of the symmetries of the radicals.	
\end{prop}

\begin{proof}

	Suppose there are no radicals in the graph \(\Gamma\), then we shall see for each edge \(A=\{x,y\}\), we may without loss of generality find \(deg(x)>1\). That is to say, \(x\) is at least in some other edge \(B=\{x,z\}\). If \(\mathcal{K}_\Gamma\) is not trivial, let \(\sigma\in \mathcal{K}_\Gamma\) such that \(\sigma\ne (1)\in \mathcal{S}_X\), and since the choice of \(A\) is arbitrary, we may simply take \(A\) to be the edge so that \(\sigma\) does not fix \(x\) and \(y\) at the same time. If \(\sigma(x)\ne x\), then to fix \(A\), \(\sigma(x)=y\), and similarly, to fix \(B\), we have \(\sigma(x)=z\), which impies that \(z=y\) and it's impossible. If \(\sigma(y)\ne y\), then to fix \(A\), we have \(\sigma(y)=x\) and \(\sigma(x)=y\), and similarly, to fix \(B\), \(\sigma(x)=z\), and hence another contradiction. 
	
	We have shown if a spanning graph has no radicals, it has trivial \(\mathcal{K}_\Gamma\). Thus, a spanning graph with nontrivial \(\mathcal{K}_\Gamma\) must have finitely many radicals. Since radicals are mutually disjoint, we see \(\mathcal{K}_\Gamma\) will be a direct product of their symmetries.
	
	In a word, if we let \(\{r_1,\,\dots,\,r_p\}\subseteq \Gamma\) denote the radicals, we see
	\[
	\mathcal{K}_\Gamma=\bigotimes_{l=1}^{p}\mathcal{S}_{r_l} 
	\]
\end{proof}

\section{Ring of Partials}
The above result tells us that the structure of \(\mathcal{K}_\Gamma\) is somewhat simple in general, and what we should aim at is the cosets of \(\mathcal{K}_\Gamma\). We shall continue to focus on spanning graphs.

The first question is: \emph{What are \(\sigma_{ij}G_i\)'s?} Well, one thing we must be aware of is that \(\sigma_{ij}\)'s are involutions of orders less than or equal to 2 and \(G_i\)'s are stabilisers of \(A_i\)'s. Such clear structures stimulate us to find out what are in \(\sigma_{ij}G_i\).

\begin{defn}[\emph{p-partial}]
	 A \(p\)-\textbf{partial} on \(X\) is defined to be an injective function \(\pi\) from a \(p\)-subset of \(X\) to \(X\) itself. A \(p\)-partial \(\pi\) can be expressed as
		\[
		\left(\begin{matrix}
		a_1&a_2&\cdots&a_p\\
		\pi_{a_1}&\pi_{a_2}&\cdots&\pi_{a_p}
		\end{matrix}
		\right)
		\]
	where \(\{a_1,\,a_2,\,\dots,\,a_p\}\subset X\) is the \(p\)-subset.

\end{defn}

\begin{prop}
	Any partial can be extended to a permutation on \(X\).	
\end{prop}
\begin{proof}
		 We at first claim that if the partial is essentially a permutation over its domain, then it can be extended to the whole \(X\), which is easy to do by making it identity everywhere outside its domain. Now we assume that if it is not a permutation yet. According to the definition, the domain and the codomain must have the same number of elements. Let \(A\) and \(B\) denote the domain and the codomain respectively, by assumption, we shall see \(\Delta_1\coloneq A\backslash B\) have the same cardinality as \(\Delta_2\coloneq B\backslash A\). By \emph{Remark} \ref{constructionofsigma}, we can construct a bijection from \(\Delta_1\) to \(\Delta_2\), which along with the original partial, will without doubt induce a permutation over \(A\bigcup B\). According to the first claim, we are done.
\end{proof}

Notice that given a \(p\)-partial, we don't have the information about where it comes from, but for two partials we can check whether they can be joint together. In other words, for a \(p\)-partial
	\[
	\left(\begin{matrix}
	a_1&a_2&\cdots&a_p\\
	j_{a_1}&j_{a_2}&\cdots&j_{a_p}
	\end{matrix}
	\right)
	\] and a \(q\)-partial
	\[
	\left(\begin{matrix}
	b_1&b_2&\cdots&b_q\\
	l_{b_1}&l_{b_2}&\cdots&l_{b_q}
	\end{matrix}
	\right)
	\] 
	we say they are \emph{uniform} if they do not collide with each other in the sense that if there is, for instance, \(a_t=b_h\), then \(j_{a_t}=l_{b_h}\), and if there is \(j_{a_t^\prime}=l_{b_h^\prime}\), then \(a_t^\prime=b_h^\prime\). 
\begin{defn}
	On the class of all the partials over \(X\), we define the following order:
	\[
	j\le l \text{ if and only if } dom(j)\subseteq dom(l),\, cod(j)\subseteq cod(l)\text{ and } l \text{ can restrict to }j
	\]
\end{defn}

\begin{prop}
Given finitely many uniform partials, there exists a unique minimum partial that can restrict to all of them. 
\end{prop}

\begin{proof}
We at first consider the case of two partials, and we continue to use the above notation. To show the existence, we would like to construct one. If their domains are disjoint, then by the definition of uniformity, their codomain must also be disjoint, and thus we may simply put them together without losing injectivity. By definition of uniformity, one should notice that on the subset \(D\coloneq dom(j)\bigcap dom(l)\), \(j|_D=l|_D\). Thus, we can at first put \(j|_{X\backslash D}\) and \(l|_{X\backslash D}\) together and then join it with \(j|_D=l|_D\). It is clear that the order does not matter. We can inductively show that the existence is true for finitely many partials. 

To see our construction is the unique minimum one, we shall see that any partial that is bigger than all of the prescribed partials must has its domain contain the union of all those domains, and codomain contain the union of all those codomains. Since on each term of the union the partial, the prescribed one and our construction are equal, the partial can restrict to our construction, and hence is bigger than our construction. If there is another minimum one, we can see it is both bigger and smaller than our construction, and so the two must be equal.
\end{proof}

\begin{defn}[\emph{Join}]
	Given two partials \(\pi_1\) and \(\pi_2\), we will call the above construction their \textbf{join}, and use \(\pi_1\vee \pi_2\) to denote it.
\end{defn}

\begin{exmp}
	Over \(X=\{1,\,2,\,\dots,\,100\}\), we consider some examples to show how partials are joint together.
	\begin{itemize}
		\item[\(a.\)] The join of
		\[
		\left(\begin{matrix}
		1&2&4&7\\
		11&34&5&1
		\end{matrix}
		\right)
		\]
		and 
		\[
		\left(\begin{matrix}
		10&20&24&17&59&50\\
		12&35&18&16&27&77
		\end{matrix}
		\right)
		\]
		will be
		\[
		\left(\begin{matrix}
		1 &2 &4&7&10&20&24&17&59&50\\
		11&34&5&1&12&35&18&16&27&77
		\end{matrix}
		\right)
		\]
		This is the case with no repetition;
		\item[\(b.\)] The join of
		\[
		\left(\begin{matrix}
		1&2\\
		11&14
		\end{matrix}
		\right)
		\]
		and 
		\[
		\left(\begin{matrix}
		11&14\\
		1&2
		\end{matrix}
		\right)
		\]
		will be 
		\[
		\left(\begin{matrix}
		11&14&1&2\\
		1&2&11&14
		\end{matrix}
		\right)
		\]
		which is essentially a permutation (while it is not yet, because we don't know what is happening outside the subset \(\{1,2,11,14\}\));
		\item[\(c.\)] The join of
		\[
		\left(\begin{matrix}
		1&2\\
		1&2
		\end{matrix}
		\right)
		\]
		and 
		\[
		\left(\begin{matrix}
		3&4\\
		3&4
		\end{matrix}
		\right)
		\]
		will be
		\[
		\left(\begin{matrix}
		1&2&3&4\\
		1&2&3&4
		\end{matrix}
		\right)
		\]
		One thing that readers must be careful here is that
		\[
		\left(\begin{matrix}
		1&2&3&4\\
		1&2&3&4
		\end{matrix}\right)\ne \text{ either } \left(\begin{matrix}
		3&4\\
		3&4
		\end{matrix}
		\right) \text{ or } \left(\begin{matrix}
		1&2\\
		1&2
		\end{matrix}
		\right)\]
		\item[\(d.\)] The join of 
		\[
		\left(\begin{matrix}
		1&2&3&4\\
		1&2&3&4
		\end{matrix}\right)
		\]
		and
		\[
		\left(\begin{matrix}
		1&45&67&49\\
		1&20&33&35
		\end{matrix}\right)
		\]
		will be
		\[
		\left(\begin{matrix}
		1&2&3&4&45&67&49\\
		1&2&3&4&20&33&35
		\end{matrix}\right)
		\]
		\item[\(e.\)] 
		\[
		\left(\begin{matrix}
		1&2&3&4&45&67&49\\
		1&2&3&4&20&33&35
		\end{matrix}\right)
		\]
		cannot get joint together with
		\[
		\left(\begin{matrix}
		1&2&3&4&45&67&49\\
		12&22&31&41&20&33&35
		\end{matrix}\right)
		\]
		\item[\(f.\)]
		\[
		\left(\begin{matrix}
		1&2&3\\
		12&22&31
		\end{matrix}\right)
		\]
		can be extended to the cycle 
		\[
		\left(\begin{matrix}
		1&2&3&12&22&31&4&5&\cdots&100\\
		12&22&1&1&2&3&4&5&\cdots&100
		\end{matrix}\right)
		\]
	\end{itemize}
\end{exmp}

We have given a lot of examples to show how partials behave, but why do we define them? The following theorem stresses its importance.

\begin{thm}
	If we write \(X=\{s_1,\,s_2,\dots,\,s_m\}\),  \(A_l=\{s_1,\,s_2\}\), and suppose  \(\{s_i,s_j\}\ne\{s_1,s_2\} \). Then all elements in \((s_1s_i)(s_2s_j)G_l\) will have
	\[
	\left(\begin{matrix}
	s_1&s_2\\
	s_i&s_j
	\end{matrix}\right)
	\]
	or
	\[
	\left(\begin{matrix}
	s_1&s_2\\
	s_j&s_i
	\end{matrix}\right)
	\]
	as their partials, and conversely, any permutation that has one of these two as its partial must be in \((s_1s_i)(s_2s_j)G_l\). \label{partialeq}
\end{thm}
\begin{proof}
	\begin{itemize}
		\item[\textbf{I.}] 	We at first assume that \(s_1\ne s_i\) and \(s_2\ne s_j\). Notice that \(G_l=\mathcal{S}_{\{s_1,s_2\}}\otimes\mathcal{S}_{\{s_3,\,\dots,\,s_m\}}\), then we have for \(g\in G_l\)
		\begin{itemize}
			\item[\(i.\)] If \(g\in \mathcal{S}_{\{s_3,\,\dots,\,s_m\}}\), then if \(s_i\) and \(s_j\) are in the same cycle
			\[
			\begin{split}
			(s_1s_i)(s_2s_j)g&=(s_1s_i)(s_2s_j)(s_i\cdots s_j\cdots)\cdots(\cdots)\\
			&=(s_1s_i\cdots s_2s_j\cdots)\cdots(\cdots)
			\end{split}
			\]
			
			And if \(s_i\) and \(s_j\) are in different cycles, then
			\[
			\begin{split}
			(s_1s_i)(s_2s_j)g&=(s_1s_i)(s_2s_j)(s_i\cdots)(s_j\cdots)\cdots(\cdots)\\
			&=(s_1s_i\cdots)(s_2s_j\cdots)\cdots(\cdots)
			\end{split}
			\]
			
			so both of them have
			\[\left(\begin{matrix}
			s_1&s_2\\
			s_i&s_j
			\end{matrix}\right)\]
			as their partials;
			
			\item[\(ii.\)] If \(g\in (s_1s_2) \mathcal{S}_{\{s_3,\,\dots,\,s_m\}}\), then if \(s_i\) and \(s_j\) are in the same cycle
			\[
			\begin{split}
			(s_1s_i)(s_2s_j)g&=(s_1s_i)(s_2s_j)(s_1s_2)(s_i\cdots s_j\cdots)\cdots(\cdots)\\
			&=(s_1s_j\cdots s_2s_i\cdots)\cdots(\cdots)
			\end{split}
			\]
			
			And if \(s_i\) and \(s_j\) are in different cycles, then
			\[
			\begin{split}
			(s_1s_i)(s_2s_j)g&=(s_1s_i)(s_2s_j)(s_1s_2)(s_i\cdots)(s_j\cdots)\cdots(\cdots)\\
			&=(s_1s_j\cdots)(s_2s_i\cdots)\cdots(\cdots)
			\end{split}
			\]
			
			so both of them have
			\[\left(\begin{matrix}
			s_1&s_2\\
			s_j&s_i
			\end{matrix}\right)\]
			as their partials;
		\end{itemize}
		\item[\textbf{II.}] Without loss of generality we may assume that \(s_1\ne s_i\) and \(s_2=s_j\), and so \(s_j\notin \{s_3,\,\dots,\,s_m\}\).
		\begin{itemize}
			\item[\(i.\)] If \(g\in \mathcal{S}_{\{s_3,\,\dots,\,s_m\}}\), then
			\[
			\begin{split}
			(s_1s_i)(s_2s_j)g&=(s_1s_i)(s_2s_j)(s_i\cdots)\cdots(\cdots)\\
			&=(s_1s_i\cdots)(s_2s_j)\cdots(\cdots)
			\end{split}
			\]	
			so
			\[\left(\begin{matrix}
			s_1&s_2\\
			s_i&s_j
			\end{matrix}\right)\]
			is its partial;
			
			\item[\(ii.\)] If \(g\in (s_1s_2)\mathcal{S}_{\{s_3,\,\dots,\,s_m\}}\), then
			\[
			\begin{split}
			(s_1s_i)(s_2s_j)g&=(s_1s_i)(s_2s_j)(s_1s_2)(s_i\cdots)\cdots(\cdots)\\
			&=(s_1s_js_2s_i\cdots)\cdots(\cdots)\\
			&=(s_1s_2s_i\cdots)\cdots(\cdots)
			\end{split}
			\]	
			so
			\[\left(\begin{matrix}
			s_1&s_2\\
			s_j&s_i
			\end{matrix}\right)\]
			is its partial.
		\end{itemize}
		\item[\textbf{III.}] For a permutation \(\sigma\in \mathcal{S}_X\), if it has
		\[
		\left(\begin{matrix}
		s_1&s_2\\
		s_i&s_j
		\end{matrix}\right)
		\]
		as its partial, then \((s_2s_j)(s_1s_i)\sigma(s_1)=s_1\), and \((s_2s_j)(s_1s_i)\sigma(s_2)=s_2\). Thus \((s_2s_j)(s_1s_i)\sigma\) must fix the edge \(A_l\), and so \((s_2s_j)(s_1s_i)\sigma\in G_l\), which implies that \(\sigma\in (s_1s_i)(s_2s_j)G_l\). The second case is similar to prove.
	\end{itemize}
\end{proof}

\begin{defn}[\emph{Partial Equivalence}]
	For \(\sigma,\,\tau\) in \(\mathcal{S}_X\), we may call them \textbf{partially equivalent} relative to \(\pi\) if \(\sigma\) and \(\tau\) share the same partial \(\pi\). In this momoent, we will use \(\sigma\simeq_{\pi}\tau\) to denote the relation.
\end{defn}

\begin{prop}
	The above definition is exactly an equivalence relation.
\end{prop}

From now on, for a partial \(\pi\) we call \([\pi]\) the \(\pi\)-\emph{class}, collecting all the permutations in \(S_X\) that have \(\pi\) as their partials. Therefore, \emph{Theorem} \ref{partialeq} in fact tells us that 
\[
(s_1s_i)(s_2s_j)G_l=\left[\left(\begin{matrix}
s_1&s_2\\
s_j&s_i
\end{matrix}\right)\right]\bigsqcup\left[\left(\begin{matrix}
s_1&s_2\\
s_i&s_j
\end{matrix}\right)\right].
\]

\noindent Now, return to the matrix
\[
\mathcal{M}_\Gamma=
\left(\begin{matrix}
\sigma_{11}G_1 & \sigma_{12}G_1 & \cdots&\sigma_{1n}G_1\\
\sigma_{21}G_2 & \sigma_{22}G_2 & \cdots&\sigma_{2n}G_2\\
\vdots&  \vdots&  \ddots&\vdots\\
\sigma_{n1}G_n & \sigma_{n2}G_n & \cdots &\sigma_{nn}G_n
\end{matrix}
\right)
\]
One must observe this matrix gives everything about the automorphisms of a given graph, so our aim is to extract its information in a more elegant way. To this end, we would like to introduce the algebraic structure of partials. In the following context, we will call \(\mathcal{M}_\Gamma\) the \emph{Canonical Matrix} of \(\Gamma\). 

\begin{defn}[\emph{Brics}]
	For the set of ground points \(X=\{s_1,\,s_2,\,\dots,\,s_m\}\), and two variables \(\mu,\nu\) over \(X\), we would call \(\left(\begin{matrix}
	\mu\\
	\nu
	\end{matrix}\right)\) a \textbf{free bric} over \(X\). If we input \(\mu=s_i\) and \(\nu=s_j\), we would call \(\left(\begin{matrix}
	s_i\\
	s_j
	\end{matrix}\right)\) a \textbf{bric} over \(X\).
\end{defn}

\noindent Here one must notice \(\left(\begin{matrix}
s_i\\
s_i
\end{matrix}\right)\ne\left(\begin{matrix}
s_i^\prime\\
s_i^\prime
\end{matrix}\right)\), if \(s_i\ne s_i^\prime\). 

\noindent\textbf{Brainstorm}: In fact, brics are \(1\)-partials over \(X\), and we want operations like \emph{addition} to say that some function share either this partial or that one, and \emph{multiplication} to integrate information of brics by using \emph{joins}. One should notice that free brics behave in the same way of brics.

We start with \(\omega\coloneq
\{\nu_\lambda\}_{\lambda\in\Lambda}\), a collection of symbols, representing free variables over \(X=\{s_1,\,\dots,\,s_m\}\). Then we define 
\[
\omega^\beta\coloneq\left\{\left(\begin{matrix}
\nu_i\\
\nu_j
\end{matrix}\right)\,;\,\nu_i,\,\nu_j\in\omega\right\}
\]
a collection of free brics over \(X\) based on \(\omega\). Now, we view elements in \(\omega^\beta\) as variables over \(\mathbb{F}_2\), and generate a collection \(\mathbb{F}_2(\omega^\beta)\) of multi-variable polynomials over this field, which is in fact a linear space over \(\mathbb{F}_2\). On this \(\mathbb{F}_2(\omega^\beta)\), we will quotient the following relations:
\begin{itemize}
	\item[\(a.\)]  For \(\pi\in \omega^\beta\), \(\pi^2=\pi\);
	\item[\(b.\)]  For \(\pi_1,\,\pi_2\in\omega^\beta\), \(\pi_1\pi_2=0\) if their upper rows are equal but the lower rows are not, or the upper row are not equal but the lower rows are.
\end{itemize}

\begin{prop}
	Let \[\mathcal{I}_1\coloneq
	<\pi^2-\pi\,;\,\pi\in\omega^\beta>\] and \[\mathcal{I}_2\coloneq<\pi_1\pi_2\,;\,\pi_1,\,\pi_2\in\omega^\beta \text{ that satisfy the properties in }b.>\]
	Then, \(\mathcal{I}_1\) and \(\mathcal{I}_2\) are ideals in \(\mathbb{F}_2(\omega^\beta)\) corresponding to the two relations above respectively, and so if we quotient the relations in \(a.\) and \(b.\), we obtain a ring.
\end{prop}

\begin{proof}
	Trivial.
\end{proof}
\noindent We denote what we will obtain after the quotient \(\Pi_\omega(X)\).

\begin{defn}[\emph{Ring of Free Partials}]
	We will call \(\Pi_\omega(X)\) the \textbf{Ring of Free Partials over}  \(X\).
\end{defn}

In fact, if we simply put \(\omega\coloneq X\), we obtain a ring generated from a ring of \(m^2\)-variable polynomials over \(\mathbb{F}_2\). 

\begin{defn}[\emph{Ring of Partials}]
	We will call this ring the \textbf{Ring of Partials over} \(X\), denoted by \(\Pi(X)\).
\end{defn}

\begin{prop}
	In \(\Pi(X)\), if we denote its multiplication by \(\ast\), and then we shall see for two brics \(\pi_1\) and \(\pi_2\)
	\[
	\pi_1\ast\pi_2=\begin{cases}
	\pi_1\vee\pi_2 & \pi_1\text{ and }\pi_2 \text{ are uniform};\\
	0 & \text{ otherwise}.
	\end{cases}
	\]
	Since the brics span the whole space, the multiplication in \(\Pi(X)\) is essentially ``\(\vee\)".
\end{prop}
\begin{proof}
	After the construction.
\end{proof}

\begin{prop}
	\[
\begin{split}
	\Pi(X)&=\left\{\sum_l\prod_{k}\pi_l^k+c,\,mod<\mathcal{I}_1,\mathcal{I}_2>\,;\, \pi_l^k\in\omega^\beta(X), c=0\text{ or }1\right\}\\
	&=\left\{\sum_{l} \bigvee_{k}\pi_l^k+c\,\,;\, 
		\pi_l^k\in \omega^\beta(X),\,c=1\text{ or }0. 
		\right\}\bigcup\{0,1\}
\end{split}
	\]
\end{prop}
\begin{proof}
	After the construction.
\end{proof}

 In fact, there is another definition of ring of partials, which is defined in a more concrete way, but the two are essentially the same. Now, let me introduce the second definition.

We start with the collection of all subsets of \(S_X\), denoted \(\mathbb{P}(S_X)\). In this collection, we define for \(A,B\in\mathbb{P}(S_X)\)
\[
\begin{split}
&A\ast B\coloneq A\bigcap B\\
&A+B\coloneq (A\backslash B)\bigcup (B\backslash A )
\end{split}
\]
It is interesting to recall the knowledge of ``Boolean Algebra" to know that \((\mathbb{P}(\mathcal{S}_X),\ast,+)\) becomes an algebra over the field \(\mathbb{F}_2\).

In \(\mathbb{P}(S_X)\), we collect all the partial classes over \(X\)
\[
\mathcal{PC}\coloneq \{[\pi]\in\mathbb{P}(S_X)\,;\,\pi \text{ is a partial over }X\}
\]
One thing we must notice is that elements in \(\mathcal{PC}\) may not be disjoint, but this doesn't matter at all. 

\vspace*{0.2cm}

\noindent Since the intersection of two subalgebras of \(\mathbb{P}(S_X)\) is still an algebra, it is reasonable to consider the smallest subalgebra of \(\mathbb{P}(S_X)\) that contain \(\mathcal{PC}\). We will call this subalgebra \emph{the second definition of the Ring of Partials over }\(X\), denoted by \(\Pi^\prime(X)\). Observe that every permutation is per se a partial, and hence any subset of \(S_X\) is a sum of partial classes, and thus \(\mathcal{PC}\) spans the whole \(\mathbb{P}(S_X)\), and so the so-called \(\Pi^\prime(X)\) is nothing but \(\mathbb{P}(S_X)\). The only difference is that, we consider it in a new way.

\begin{thm}
	There is an isomorphism from \(\Pi(X)\) to \(\Pi^\prime(X)\). \label{isotwodef}
\end{thm}

\begin{proof}
	We define
	\[
	\begin{split}
	\Lambda: \Pi(X)&\longrightarrow\Pi^\prime(X)\\
	\sum\bigvee\pi &\longmapsto \sum\left[\bigvee\pi\right]\\
	0 &\longmapsto \emptyset\\
	1 &\longmapsto S_X
	\end{split}
	\]
	and it is easy to check this is an isomorphism.
\end{proof}

\section{Graphs and Matrices}

From previous section we have defined over \(X\) the ring of partials. Here we will not distinguish the two definitions.

Suppose we have a graph \(\Gamma=\{A_1,\,\dots,\,A_n\}\), and \(A_l=\{\mu_l,\,\nu_l\}\) for each \(1\le l\le n\). Here \(\mu_l\) and \(\nu_l\) are variables over \(X\), and so we can use these variables to construct free brics. That is to say, we may construct a ring of free partials by \(\Gamma\), which is denoted by \(\Pi_\Gamma\). If we input the values of \(X\) to the variables, then \(\Pi_\Gamma\) will be well deformed into \(\Pi(X)\), and so if we have the exact data of the graph, we will simply consider \(\Pi(X)\) as our ambient ring. But for the rest of the time, we would still use \(\mu_l,\,\nu_l\) to denote the elements in \(A_l\) for each \(l\).

Recall the matrix 
\[
\begin{split}
\mathcal{M}_\Gamma&=
\left(\begin{matrix}
\sigma_{11}G_1 & \sigma_{12}G_1 & \cdots&\sigma_{1n}G_1\\
\sigma_{21}G_2 & \sigma_{22}G_2 & \cdots&\sigma_{2n}G_2\\
\vdots&  \vdots&  \ddots&\vdots\\
\sigma_{n1}G_n & \sigma_{n2}G_n & \cdots &\sigma_{nn}G_n
\end{matrix}
\right)\\
&={\small\left(\begin{matrix}
	\left[\left(\begin{matrix}
	\mu_1&\nu_1\\
	\mu_1&\nu_1
	\end{matrix}\right)\right]\bigsqcup\left[\left(\begin{matrix}
	\mu_1&\nu_1\\
	\nu_1&\mu_1
	\end{matrix}\right)\right] &  \cdots&\left[\left(\begin{matrix}
	\mu_1&\nu_1\\
	\mu_n&\nu_n
	\end{matrix}\right)\right]\bigsqcup\left[\left(\begin{matrix}
	\mu_1&\nu_1\\
	\nu_n&\mu_n
	\end{matrix}\right)\right]\\
	\left[\left(\begin{matrix}
	\mu_2&\nu_2\\
	\mu_1&\nu_1
	\end{matrix}\right)\right]\bigsqcup\left[\left(\begin{matrix}
	\mu_2&\nu_2\\
	\nu_1&\mu_1
	\end{matrix}\right)\right] &  \cdots&\left[\left(\begin{matrix}
	\mu_2&\nu_2\\
	\mu_n&\nu_n
	\end{matrix}\right)\right]\bigsqcup\left[\left(\begin{matrix}
	\mu_2&\nu_2\\
	\nu_n&\mu_n
	\end{matrix}\right)\right]\\
	\vdots&    \ddots&\vdots\\
	\left[\left(\begin{matrix}
	\mu_n&\nu_n\\
	\mu_1&\nu_1
	\end{matrix}\right)\right]\bigsqcup\left[\left(\begin{matrix}
	\mu_n&\nu_n\\
	\nu_1&\mu_1
	\end{matrix}\right)\right] &  \cdots &\left[\left(\begin{matrix}
	\mu_n&\nu_n\\
	\mu_n&\nu_n
	\end{matrix}\right)\right]\bigsqcup\left[\left(\begin{matrix}
	\mu_n&\nu_n\\
	\nu_n&\mu_n
	\end{matrix}\right)\right]
	\end{matrix}
	\right)}
\end{split}
\]  

\noindent Now, if we define 
\[
\alpha_j^i \coloneq \left(\begin{matrix}
\mu_i\\
\mu_j
\end{matrix}\right) \text{ and }
\beta_j^i \coloneq \left(\begin{matrix}
\nu_i\\
\nu_j
\end{matrix}\right)
\] 
and
\[
\zeta_j^i \coloneq \left(\begin{matrix}
\nu_i\\
\mu_j
\end{matrix}\right) \text{ and }
\gamma_j^i \coloneq \left(\begin{matrix}
\mu_i\\
\nu_j
\end{matrix}\right)
\]

\noindent With the abuse of notations, we can simply write
\begin{equation}
\mathcal{M}_\Gamma={\small\left(\begin{matrix}
	\alpha_1^1\beta_1^1+\gamma_1^1\zeta_1^1 & \alpha_2^1\beta_2^1+\gamma_2^1\zeta_2^1 & \cdots&\alpha_n^1\beta_n^1+\gamma_n^1\zeta_n^1\\
	\alpha_1^2\beta_1^2+\gamma_1^2\zeta_1^2 & \alpha_2^2\beta_2^2+\gamma_2^2\zeta_2^2 & \cdots&\alpha_n^2\beta_n^2+\gamma_n^2\zeta_n^2\\
	\vdots&  \vdots&  \ddots&\vdots\\
	\alpha_1^n\beta_1^n+\gamma_1^n\zeta_1^n & \alpha_2^n\beta_2^n+\gamma_2^n\zeta_2^n & \cdots &\alpha_n^n\beta_n^n+\gamma_n^n\zeta_n^n
	\end{matrix}
	\right)} \label{canonicalmatrix}
\end{equation}
and thus we obtain a square matrix over the ring \(\Pi(X)\). When talking about square matrices, we may come up with its determinant. What is the determinant of this square matrix \(\mathcal{M}_\Gamma\)?

\noindent\textbf{Brainstorm~\cite{n2}:} Well, recall that \(\Pi(X)\) is also a linear space over \(\mathbb{F}_2\), and so we don't distinguish \(1\) and \(-1\) in \(\Pi(X)\). Then if we consider a general square \((a_{i,j})_{n\times n}\), and assume it has some corresponding determinant, by \emph{Laplace Formula}, this determinant must satisfy
\[
\begin{split}
\det((a_{i,j})_{n\times n})&=\sum_{j=1}^{n} (-1)^{i+j} a_{i,j} M_{i,j}\\
&=\sum_{j=1}^{n}  a_{i,j} M_{i,j}
\end{split}
\]
for some \(1\le i\le n\). Here \((-1)^{i+j}M_{i,j}\) is the \emph{cofactor} of \(a_{i,j}\). In particular, we use this formula inductively from \(i=1\) to \(i=n\), and in this way we shall define

\begin{thm}[\emph{Determinant Definition}]
	\[
	\begin{split}
	\det:\textbf{M}_n\left(\Pi(X)\right)&\longrightarrow\Pi(X)\\
	(a_{i,j})_{n\times n}&\longmapsto \sum_{j\in \mathcal{S}_n} \prod_{i=1}^{n} a_{i,j_i}
	\end{split}
	\]
	is indeed a determinant. That is to say, it satisfies: \textbf{multilinearity}, \textbf{alternating property} and \textbf{identity-determinant 1}. Here \(\textbf{M}_n\left(\Pi(X)\right)\) collects all the \(n\times n\) matrices with entries in \(\Pi(X)\).
\end{thm}

\begin{proof}
	\begin{itemize}
		\item[1.] (\emph{Identity has determinant 1}) For the identity matrix
		\[
		\left(\begin{matrix}
		1 & 0 & \cdots&0\\
		0 & 1 & \cdots&0\\
		\vdots&  \vdots&  \ddots&\vdots\\
		0 & 0 & \cdots &1
		\end{matrix}\right)
		\]
		we have \(\displaystyle \prod_{i=1}^{n} a_{i,j_i}\) is not 0 if and only if \(j\) is an identity in \(\mathcal{S}_n\), and so the determinant is exactly \(1^n=1\);
		\item[2.] (\emph{Multilinearity})  We consider two matrices 
		\[
		\begin{split}
		\left(\alpha_1,\,\alpha_2,\,\dots,\,\alpha_n\right) &\coloneq (a_{i,j})_{n\times n}\\
		\left(\beta_1,\,\beta_2,\,\dots,\,\beta_n\right) &\coloneq (b_{i,j})_{n\times n}
		\end{split}
		\]
		Then
		\[
		\begin{split}
		&\;\det\left(\left(\alpha_1,\,\dots,\,\alpha_l+\beta_l,\,\dots,\,\alpha_n\right)\right)\\=&\sum_{j\in\mathcal{S}_n}\left(a_{j^{-1}(l),l}+b_{j^{-1}(l),l}\right)\prod_{i\ne j^{-1}(l)}a_{i,j_i}\\
		=&\sum_{j\in\mathcal{S}_n}a_{j^{-1}(l),l}\prod_{i\ne j^{-1}(l)}a_{i,j_i}+\sum_{j\in\mathcal{S}_n}b_{j^{-1}(l),l}\prod_{i\ne j^{-1}(l)}a_{i,j_i}\\
		=&\det\left(\left(\alpha_1,\,\alpha_2,\,\dots,\,\alpha_l,\,\dots,\,\alpha_n\right)\right)+\det\left(\left(\alpha_1,\,\alpha_2,\,\dots,\,\beta_l,\,\dots,\,\alpha_n\right)\right)
		\end{split}
		\]
		for all \(1\le l\le n\);
		\item[3.] (\emph{Alternating Property}) We shall see if \(\sigma=(ts)\in S_X\), then
		\[
		\begin{split}
		\det((a_{\sigma(i),j})_{n\times n})&= \sum_{j\in \mathcal{S}_n} \prod_{i=1}^{n} a_{\sigma(i),j_i}\\
		&=\sum_{j\in \mathcal{S}_n} \prod_{\sigma(i)=1}^{n} a_{\sigma(i),j_i}\\
		&=\sum_{j\in \mathcal{S}_n} \prod_{i=1}^{n} a_{i,j\circ\sigma(i)}\\
		&=\sum_{ j\circ\sigma\in \mathcal{S}_n} \prod_{i=1}^{n} a_{i,j_i}\\
		&=\det((a_{i,j})_{n\times n})
		\end{split}
		\]
		Similarly we can show it for columns.
	\end{itemize}
\end{proof}

If we suppose \(\Gamma\) is a spanning graph, then now we are able to talk about its automorphism group. Assume we already know the data of \(\Gamma\), and so we can talk about things in the ring \(\Pi(X)\).

\begin{thm}
	With the abuse of notations, we may say for \(\Gamma\) a spanning graph over \(X\)
	\[
	Aut(\Gamma)=\det(\mathcal{M}_\Gamma)
	\]
	here \(\mathcal{M}_\Gamma\) is the canonical matrix of the spanning graph \(\Gamma\).
\end{thm}

\begin{proof}
 We have written \(\mathcal{M}_\Gamma\) as in \emph{Equation} (\ref{canonicalmatrix}). By \emph{Proposition} \ref{determinantuse} and \emph{Theorem} \ref{isotwodef}, the result is easy to see.
\end{proof}

\section{Automorphism Group of A Given Hypergraph}

\subsection{Arbitrary Graphs}
One may find in the above theorem, we don't really need a graph to be spaning. In fact, the above theorem can be naturally extended to a much wider stage. We will give the general version for arbitrary hypergraphs over the prescribed \(X\).

To begin with, we will extend the theorem to arbitrary graphs.

\begin{thm}
	With the abuse of notations, we may say for \(\Gamma\) an arbitrary graph over \(X\)
	\[
	Aut(\Gamma)=\det(\mathcal{M}_\Gamma)
	\]
	here \(\mathcal{M}_\Gamma\) is the canonical matrix of the graph \(\Gamma\).
\end{thm}

\begin{proof}
	This result is natural because \emph{Proposition} \ref{determinantuse} and \emph{Theorem} \ref{isotwodef} do not require a graph to be spanning.
\end{proof}

\subsection{Homogeneous Hypergraphs}
For a \(k\)-homogeneous hypergraph \(\Gamma=\{A_1,\,\dots,\,A_n\}\), if we write \(A_l=\{a_l^1,a_l^2,\,\dots,\,a_l^k\}\), we define the canonical matrix of \(\Gamma\) to be 
\[
\begin{split}
\mathcal{M}_\Gamma\coloneq
\left(\begin{matrix}
\left(\begin{matrix}
A_1\\
A_1
\end{matrix}\right)&\left(\begin{matrix}
A_1\\
A_2
\end{matrix}\right)&\cdots&\left(\begin{matrix}
A_1\\
A_n
\end{matrix}\right)\\
\left(\begin{matrix}
A_2\\
A_1
\end{matrix}\right)&\left(\begin{matrix}
A_2\\
A_2
\end{matrix}\right)&\cdots&\left(\begin{matrix}
A_2\\
A_n
\end{matrix}\right)\\
\vdots&\vdots&\ddots&\vdots\\
\left(\begin{matrix}
A_n\\
A_1
\end{matrix}\right)&\left(\begin{matrix}
A_n\\
A_2
\end{matrix}\right)&\cdots&\left(\begin{matrix}
A_n\\
A_n
\end{matrix}\right)
\end{matrix}\right)
\end{split}
\]
here we mean 
\[
\left(\begin{matrix}
A_i\\
A_j
\end{matrix}\right)\coloneq
\sum_{\sigma\in\mathcal{S}_k}\left(\begin{matrix}
a_i^1&a_i^2&\cdots&a_i^k\\
a_j^{\sigma(1)}&a_j^{\sigma(2)}&\cdots&a_j^{\sigma(k)}
\end{matrix}\right)
\]
for each \(1\le i,j\le n\).

With similar arguments, we can show the theorem is also true for arbitrary homogeneous hypergraphs.

\subsection{Arbitrary Hypergraphs}
Suppose that we have \(X=\{s_1,\,\dots,\,s_m\}\), and an arbitrary hypergraph \(\Gamma\). We would like to write \(\displaystyle\Gamma=\bigsqcup_{k=1}^m \Gamma_k\), where each \(\Gamma_k\) is the \(k\)-section of \(\Gamma\). Let \(\mathcal{M}_k\) be the canonical matrix of each \(\Gamma_k\) when seen as homogeneous hypergraphs over \(X\). Since there is no transition between different sections, we shall define the canonical matrix for \(\Gamma\)
\[
\mathcal{M}_\Gamma\coloneq
\left(\begin{matrix}
\mathcal{M}_1 & 0 & \cdots&0\\
0 & \mathcal{M}_2 & \cdots&0\\
\vdots&  \vdots&  \ddots&\vdots\\
0 & 0 & \cdots &\mathcal{M}_m
\end{matrix}\right)
\]

\begin{thm}
	\(\det(\mathcal{M}_\Gamma)=Aut(\Gamma)\)
\end{thm}

\begin{proof}
	By linear algebra, one can see
	\[
	\det(\mathcal{M})=\prod_{k=1}^m \det(\mathcal{M}_k)
	\]
	thus the result is proved by using \emph{Theorem} \ref{sectionintersection}. 
\end{proof}

\section{Ring of Isopartials}

When talking about automorphism groups of graphs, one often comes up with the isomorphisms between different graphs. Now, we are going to find a way to describe these isomorphisms by using our knowledge of automorphisms.

To begin with, we shall define something similar to partials. In this section, \(X\coloneq\{s_1,\,s_2,\,\dots,\,s_m\}\) and \(Y\coloneq
\{t_1,\,t_2,\,\dots,\,t_m\}\) will be our ambient ground sets.

\begin{defn}[\emph{Isopartials}]
	A \(p\)-\textbf{isopartial} is an injective function from a \(p\)-subset of \(X\) to \(Y\). The \textbf{order} among isopartials is defined similar to that of partials.
\end{defn}

\begin{prop}
	Any \(p\)-isopartial can be extended to a bijective function from \(X\) to \(Y\).
\end{prop}

\begin{proof}
	Let \(\pi\) be the isopartial, and if we define \(j\) a bijective function from \(X\) to \(Y\), then \(j\circ\pi\) becomes a \(p\)-partial on \(X\). Let \(\overline{j\circ\pi}\) denote the smallest extension of this partial, and we may check that \(j^{-1}\circ\overline{j\circ\pi}\) is in fact a bijection that dominates \(\pi\).
\end{proof}

For two isopartials, we can similarly define \emph{uniformity}, which is similar to that of partials. A finite collection of isopartials is called \emph{uniform} if they are pairwise uniform. 

\begin{prop}
	Given finitely many uniform isopartials, there exists a unique minimum isopartial that can restrict to all of them. 
\end{prop}

\begin{proof}
	Routine proof.
\end{proof}

\begin{defn}
	We shall define the above isopartial the \textbf{join} of the collection. And its symbol will be exactly that of partials.
\end{defn}

To construct what we shall call \emph{the Ring of Isopartials}, we would like to at first recall the following commutative diagram
\[
\xymatrix{
Y\ar[r]^j  &X\\
X\ar[u]^\pi \ar[ru]_{j\circ\pi} &}
\]
where \(j\) is a bijective function from \(Y\) to \(X\) and \(\pi\) is any bijection from \(X\) to \(Y\). Now, if we denote \(S_{X,Y}\) the collection of all bijections from \(X\) to \(Y\), we shall see \(j\) induces the following embedding
\[
\begin{split}
j^\ast: S_{X,Y} &\longrightarrow S_X\\
        \pi\quad&\longmapsto  j\circ\pi
\end{split}
\]
In fact, \(j^\ast\) is also a function on the power set of \(S_{X,Y}\):
\[
\begin{split}
j^\ast: \mathbb{P}(S_{X,Y}) &\longrightarrow \mathbb{P}(S_X)\\
\{\pi\}\quad&\longmapsto  \{j\circ\pi\}
\end{split}
\]
According to previous arguments, \(\mathbb{P}(S_X)\) is essentially \(\Pi(X)\), and thus \((j^\ast)^{-1}\) gives a structure on \(\mathbb{P}(S_{X,Y})\), which is exactly what we want. 

\begin{prop}
\(\displaystyle	(j^\ast)^{-1}(\Pi(X))\) is independent of the choice of \(j\).
\end{prop}
\begin{proof}
	It's a routine proof.
\end{proof}

\begin{defn}[\emph{Ring of Isopartials}]
	In the following context, we will call \(\displaystyle	(j^\ast)^{-1}(\Pi(X))\) \textbf{the Ring of Isopartials} from \(X\) to \(Y\), and will use \(\Pi(X,Y)\) to denote it. 
\end{defn}

\begin{rmk}
	One can see that there is a more direct way to consider the graph isomorphism problem by simply thinking of it in the ring \(\Pi(X\sqcup Y)\), of which \(\Pi(X,Y)\) is a subset. But here we regard both as the same, because they are special cases of each other.
\end{rmk}

\begin{defn}[\emph{Canonical Transformation}]
	Let \(\Gamma_1\coloneq\{A_1,\,\dots,\,A_n\}\) and \(\Gamma_2\coloneq
	\{B_1,\,\dots,\,B_n\}\) be two graphs on \(X\) and \(Y\) respectively, where \(A_i=\{a_i^1,a_i^2\}\) and \(B_i=\{b_i^1,b_i^2\}\). A \textbf{canonical transformation} from \(\Gamma_1\) to \(\Gamma_2\) will be the following matrix
\[
\begin{split}
\mathcal{M}_{\Gamma_1,\Gamma_2}\coloneq
\left(\begin{matrix}
\left(\begin{matrix}
A_1\\
B_1
\end{matrix}\right)&\left(\begin{matrix}
A_1\\
B_2
\end{matrix}\right)&\cdots&\left(\begin{matrix}
A_1\\
B_n
\end{matrix}\right)\\
\left(\begin{matrix}
A_2\\
B_1
\end{matrix}\right)&\left(\begin{matrix}
A_2\\
B_2
\end{matrix}\right)&\cdots&\left(\begin{matrix}
A_2\\
B_n
\end{matrix}\right)\\
\vdots&\vdots&\ddots&\vdots\\
\left(\begin{matrix}
A_n\\
B_1
\end{matrix}\right)&\left(\begin{matrix}
A_n\\
B_2
\end{matrix}\right)&\cdots&\left(\begin{matrix}
A_n\\
B_n
\end{matrix}\right)
\end{matrix}\right)
\end{split}
\]
here we mean 
\[
\left(\begin{matrix}
A_i\\
B_j
\end{matrix}\right)\coloneq
\left(\begin{matrix}
a_i^1&a_i^2\\
b_j^{1}&b_j^{2}
\end{matrix}\right)+\left(\begin{matrix}
a_i^1&a_i^2\\
b_j^{2}&b_j^{1}
\end{matrix}\right)
\]
\end{defn}

\begin{thm}
	Let \(G_i^1\) denote the stabiliser of \(A_i\) and \(G_i^2\) that of \(B_i\), and assume \(\sigma_{ij}\) is a bijection that brings \(A_i\) to \(B_j\), then we have that \(\sigma_{ij} G_i^1=G_j^2\sigma_{ij}\) collects all the bijections that bring \(A_i\) to \(B_j\). Moreover, we have
	\[
	Iso(\Gamma_1,\Gamma_2)=\bigcap_{i}\bigcup_j \sigma_{ij} G_i^1
	\]
\end{thm}
\begin{proof}
	Routine proof.
\end{proof}
	\begin{thm}
		With the abuse of notations, we may say for \(\Gamma_1\) and \(\Gamma_2\) arbitrary graphs over \(X\) and \(Y\) respectively
		\[
		Iso(\Gamma_1,\Gamma_2)=\det(\mathcal{M}_{\Gamma_1,\Gamma_2})
		\]
		here \(\mathcal{M}_{\Gamma_1,\Gamma_2}\) is the canonical transformation from \(\Gamma_1\) to \(\Gamma_2\).
	\end{thm}
	
	\begin{proof}
		It is natural to see.
	\end{proof}

\section{Calculations}

We now go back to what we have in \emph{Section 5}. Our goal now is to find a method to calculate the determinant of a given matrix in the ring \(\Pi(X)\), especially for the canonical matrices. 

\begin{defn}[\emph{Polypartial}]
	We will call elements in the ring \(\Pi(X)\) \textbf{polypartials}. One must be able to tell the difference between partials and polypartials. For convenience, in the rest of this article, \(\pi\) either with subscripts or with superscripts will simply denote partials.
\end{defn}

\begin{prop}
	Let \(\kappa\) be a polypartial, then for any positive integer \(n\), we have \(\kappa^n=\kappa\).
\end{prop}

\begin{proof}
	By previous theorem, we may write \(\kappa=\pi_1+\dots+\pi_s\). Notice
	\[
	\begin{split}
	(\pi_1+\dots+\pi_s)^2&=\pi_1^2+\dots+\pi_s^2 + \sum_{1\le i\ne j\le s} 2\pi_i\pi_j\\
	                    &=\pi_1^2+\dots+\pi_s^2\\
	                    &=\pi_1+\dots+\pi_s
	\end{split}
	\]
	thus, inductively one can show that \(\kappa^n=\kappa\).
\end{proof}

We want to mimic what we have learnt in \emph{Linear Algebra}, and start to consider the squares in \textit{\(\textbf{M}_{n}\)}\(\left(\Pi(X)\right)\). On \textit{\(\textbf{M}_{n}\)}\(\left(\Pi(X)\right)\) we can still define addition by adding entries, and define multiplication as we have down to matrices over the reals. Therefore \textit{\(\textbf{M}_{n}\)}\(\left(\Pi(X)\right)\) becomes a ring. We can do `scalar' multiplication by multiplying polypartials to the entries and thus this ring becomes a module over \(\Pi(X)\). Many of the properties of matrices over the reals remain true here. We list them in the following theorem.

\begin{thm}
	\(A=(a_{ij})_{n\times n},B=(b_{ij})_{n\times n}\) are in \(\textbf{M}_{n}\left(\Pi(X)\right)\), and \(A^\mathnormal{T}\) is the transpose of \(A\), then
	\begin{itemize}
	\item[1.] \(\det(I_n)=1\);
	\item[2.] \(\det(A^\mathnormal{T})=\det(A)\);
	\item[3.] \(\det(AB)=\det(A)\det(B)\);
    \item[4.] 	
	\(\det((\kappa a_{ij})_{n\times n})=\det(\kappa( a_{ij})_{n\times n})=\kappa\det((a_{ij})_{n\times n});\)

	\item[5.] If \(A\) is triangular, then 
	\[
	\det(A)=\prod_{i=1}^n a_{i,i}\;;
	\]
	\item[6.] \(\det\) is multilinear;
	\item[7.] \(\det\) has alternating property.
	\end{itemize}
\end{thm}

\begin{proof}
\begin{itemize}
	\item[1.] We have shown 1, 6, 7 before;
	\item[2.] By definition
	\[
	\begin{split}
	\det(A^\mathnormal{T})&=\det((a_{ji})_{n\times n})\\
	                      &=\sum_{j\in S_n}\prod_{i=1}^{n}  a_{j_i,i}\\
	                      &=\sum_{j\in S_n}\prod_{i=1}^{n}  a_{i,j^{-1}(i)}\\
	                      &=\sum_{j\in S_n}\prod_{i=1}^{n}  a_{i,j_i}\\
	                      &=\det(A);
	\end{split}
	\]
	\item[3.] Define \(d_A(B)\coloneq \det(AB)\), we see \(d_A\) is an alternating multilinear function, and since the dimension of the space of such alternating multilinear functions is 1, we may suppose \(d_A(B)=\kappa\det(B)\). Let \(B=I_n\), we see \(\kappa=\det(A)\), thus we have that \(\det(AB)=\det(A)\det(B)\);
	\item[4.] By definition	\[
	\begin{split}
	\det((\kappa a_{ij})_{n\times n})&= \sum_{j\in S_n} \prod_{i=1}^{n} \kappa a_{ij_i}\\
	&=\sum_{j\in S_n}\kappa^n  \prod_{i=1}^{n}  a_{ij_i}\\
	&=\kappa^n \det(( a_{ij})_{n\times n})\\
	&=\kappa\det((a_{ij})_{n\times n}).
	\end{split}
	\]
	\item[5.] It is clear to see.
\end{itemize}
\end{proof}

Now, recall the the canonical matrix is of the shape

\[
\mathcal{M}_\Gamma={\small\left(\begin{matrix}
	\alpha_1^1\beta_1^1+\gamma_1^1\zeta_1^1 & \alpha_2^1\beta_2^1+\gamma_2^1\zeta_2^1 & \cdots&\alpha_n^1\beta_n^1+\gamma_n^1\zeta_n^1\\
	\alpha_1^2\beta_1^2+\gamma_1^2\zeta_1^2 & \alpha_2^2\beta_2^2+\gamma_2^2\zeta_2^2 & \cdots&\alpha_n^2\beta_n^2+\gamma_n^2\zeta_n^2\\
	\vdots&  \vdots&  \ddots&\vdots\\
	\alpha_1^n\beta_1^n+\gamma_1^n\zeta_1^n & \alpha_2^n\beta_2^n+\gamma_2^n\zeta_2^n & \cdots &\alpha_n^n\beta_n^n+\gamma_n^n\zeta_n^n
	\end{matrix}
	\right)}
\]
To calculate the determinant of this matrix, one can directly use the \emph{Leibniz Formula} to obtain the result step by step. But here we would like to recall our original idea, \emph{Theorem} \ref{interunion}. 

Observe that on each row of the matrix \(\mathcal{M}_\Gamma\), entries in different columns are disjoint, and thus their union is exactly their sum. 

\begin{defn}[\emph{Initiators and Terminators}]
For the \(i\)-th row \(\left(A_i\right)\), we define the \(i\)-th \textbf{initiator} \(\hat{\tau}_i\coloneq(\alpha_1^i\beta_1^i+\gamma_1^i\zeta_1^i) + (\alpha_2^i\beta_2^i+\gamma_2^i\zeta_2^i) + \cdots+(\alpha_n^i\beta_n^i+\gamma_n^i\zeta_n^i)\), and for the \(j\)-th column \(\left(A_j\right)\), we define the \(j\)-th \textbf{terminator} \(\check{\tau}_j\coloneq(\alpha_j^1\beta_j^1+\gamma_j^1\zeta_j^1) + (\alpha_j^2\beta_j^2+\gamma_j^2\zeta_j^2) + \cdots+(\alpha_j^n\beta_j^n+\gamma_j^n\zeta_j^n)\).
\end{defn}

\begin{thm}
	\[
	\begin{split}
	\det(\mathcal{M}_\Gamma)&=\hat{\tau}_1\hat{\tau}_2\cdots\hat{\tau}_n\\
	                        &=\check{\tau}_1\check{\tau}_2\cdots\check{\tau}_n
	\end{split}
	\]
\end{thm}

\begin{proof}
	The first equality is after \emph{Theorem} \ref{interunion}. The second product is in fact equal to the determinant of the transpose of \(\mathcal{M}_\Gamma\) by the same theorem and since determinants are invariant under transposition, we see the equality holds.
\end{proof}

This will be our main method to calculate the determinant of canonical matrices. We will give an example to illustrate the process.

\begin{exmp}
We now consider the automorphism group of the following graph, denoted \(\Gamma\):
	
	\begin{center}
		\includegraphics[width=5cm]{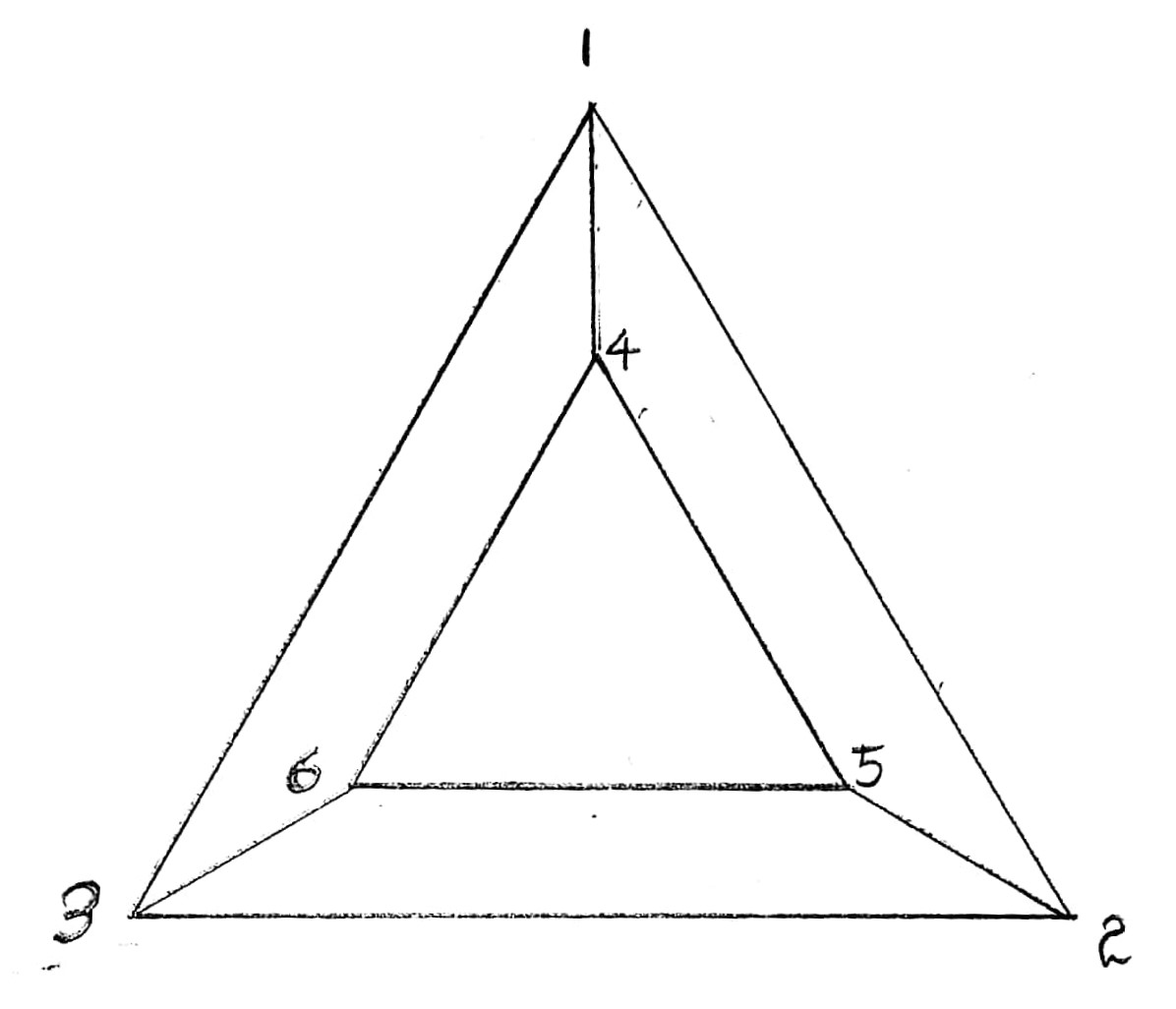}
	\end{center}
For simplicity we write the edges as \(ij\coloneq\{i,j\}\), and thus \(ij=ji\), and we define 
\[\left[\begin{matrix}
a&b\\
c&d
\end{matrix}\right]\coloneq \left(\begin{matrix}
a&b\\
c&d
\end{matrix}\right)+\left(\begin{matrix}
a&b\\
d&c
\end{matrix}\right)
\]
Now, their initiators are listed as follows:

\begin{itemize}
	\item[Edge 12] \[
	\begin{split}
	\hat{\tau}_{12}=&\left[\begin{matrix}
	1&2\\
	1&2
	\end{matrix}\right]+\left[\begin{matrix}
	1&2\\
	2&3
	\end{matrix}\right]+\left[\begin{matrix}
	1&2\\
	1&3
	\end{matrix}\right]+\left[\begin{matrix}
	1&2\\
	1&4
	\end{matrix}\right]+\left[\begin{matrix}
	1&2\\
	2&5
	\end{matrix}\right]\\&+\left[\begin{matrix}
	1&2\\
	3&6
	\end{matrix}\right]+\left[\begin{matrix}
	1&2\\
	4&5
	\end{matrix}\right]+\left[\begin{matrix}
	1&2\\
	5&6
	\end{matrix}\right]+\left[\begin{matrix}
	1&2\\
	4&6
	\end{matrix}\right]
	\end{split}
	\]
	
		\item[Edge 23] \[
		\begin{split}
		\hat{\tau}_{23}=&\left[\begin{matrix}
		2&3\\
		1&2
		\end{matrix}\right]+\left[\begin{matrix}
		2&3\\
		2&3
		\end{matrix}\right]+\left[\begin{matrix}
		2&3\\
		1&3
		\end{matrix}\right]+\left[\begin{matrix}
		2&3\\
		1&4
		\end{matrix}\right]+\left[\begin{matrix}
		2&3\\
		2&5
		\end{matrix}\right]\\&+\left[\begin{matrix}
		2&3\\
		3&6
		\end{matrix}\right]+\left[\begin{matrix}
		2&3\\
		4&5
		\end{matrix}\right]+\left[\begin{matrix}
		2&3\\
		5&6
		\end{matrix}\right]+\left[\begin{matrix}
		2&3\\
		4&6
		\end{matrix}\right]
		\end{split}
		\]	
		
		\item[Edge 13] \[
		\begin{split}
		\hat{\tau}_{13}=&\left[\begin{matrix}
		1&3\\
		1&2
		\end{matrix}\right]+\left[\begin{matrix}
		1&3\\
		2&3
		\end{matrix}\right]+\left[\begin{matrix}
		1&3\\
		1&3
		\end{matrix}\right]+\left[\begin{matrix}
		1&3\\
		1&4
		\end{matrix}\right]+\left[\begin{matrix}
		1&3\\
		2&5
		\end{matrix}\right]\\&+\left[\begin{matrix}
		1&3\\
		3&6
		\end{matrix}\right]+\left[\begin{matrix}
		1&3\\
		4&5
		\end{matrix}\right]+\left[\begin{matrix}
		1&3\\
		5&6
		\end{matrix}\right]+\left[\begin{matrix}
		1&3\\
		4&6
		\end{matrix}\right]
		\end{split}
		\]	
		
		\item[Edge 14] \[
		\begin{split}
		\hat{\tau}_{14}=&\left[\begin{matrix}
		1&4\\
		1&2
		\end{matrix}\right]+\left[\begin{matrix}
		1&4\\
		2&3
		\end{matrix}\right]+\left[\begin{matrix}
		1&4\\
		1&3
		\end{matrix}\right]+\left[\begin{matrix}
		1&4\\
		1&4
		\end{matrix}\right]+\left[\begin{matrix}
		1&4\\
		2&5
		\end{matrix}\right]\\&+\left[\begin{matrix}
		1&4\\
		3&6
		\end{matrix}\right]+\left[\begin{matrix}
		1&4\\
		4&5
		\end{matrix}\right]+\left[\begin{matrix}
		1&4\\
		5&6
		\end{matrix}\right]+\left[\begin{matrix}
		1&4\\
		4&6
		\end{matrix}\right]
		\end{split}
		\]	
		
		\item[Edge 25] \[
		\begin{split}
		\hat{\tau}_{25}=&\left[\begin{matrix}
		2&5\\
		1&2
		\end{matrix}\right]+\left[\begin{matrix}
		2&5\\
		2&3
		\end{matrix}\right]+\left[\begin{matrix}
		2&5\\
		1&3
		\end{matrix}\right]+\left[\begin{matrix}
		2&5\\
		1&4
		\end{matrix}\right]+\left[\begin{matrix}
		2&5\\
		2&5
		\end{matrix}\right]\\&+\left[\begin{matrix}
		2&5\\
		3&6
		\end{matrix}\right]+\left[\begin{matrix}
		2&5\\
		4&5
		\end{matrix}\right]+\left[\begin{matrix}
		2&5\\
		5&6
		\end{matrix}\right]+\left[\begin{matrix}
		2&5\\
		4&6
		\end{matrix}\right]
		\end{split}
		\]	
		
		\item[Edge 36] \[
		\begin{split}
		\hat{\tau}_{36}=&\left[\begin{matrix}
		3&6\\
		1&2
		\end{matrix}\right]+\left[\begin{matrix}
		3&6\\
		2&3
		\end{matrix}\right]+\left[\begin{matrix}
		3&6\\
		1&3
		\end{matrix}\right]+\left[\begin{matrix}
		3&6\\
		1&4
		\end{matrix}\right]+\left[\begin{matrix}
		3&6\\
		2&5
		\end{matrix}\right]\\&+\left[\begin{matrix}
		3&6\\
		3&6
		\end{matrix}\right]+\left[\begin{matrix}
		3&6\\
		4&5
		\end{matrix}\right]+\left[\begin{matrix}
		3&6\\
		5&6
		\end{matrix}\right]+\left[\begin{matrix}
		3&6\\
		4&6
		\end{matrix}\right]
		\end{split}
		\]	
		
		\item[Edge 45] \[
		\begin{split}
		\hat{\tau}_{45}=&\left[\begin{matrix}
		4&5\\
		1&2
		\end{matrix}\right]+\left[\begin{matrix}
		4&5\\
		2&3
		\end{matrix}\right]+\left[\begin{matrix}
		4&5\\
		1&3
		\end{matrix}\right]+\left[\begin{matrix}
		4&5\\
		1&4
		\end{matrix}\right]+\left[\begin{matrix}
		4&5\\
		2&5
		\end{matrix}\right]\\&+\left[\begin{matrix}
		4&5\\
		3&6
		\end{matrix}\right]+\left[\begin{matrix}
		4&5\\
		4&5
		\end{matrix}\right]+\left[\begin{matrix}
		4&5\\
		5&6
		\end{matrix}\right]+\left[\begin{matrix}
		4&5\\
		4&6
		\end{matrix}\right]
		\end{split}
		\]	
		
		\item[Edge 56] \[
		\begin{split}
		\hat{\tau}_{56}=&\left[\begin{matrix}
		5&6\\
		1&2
		\end{matrix}\right]+\left[\begin{matrix}
		5&6\\
		2&3
		\end{matrix}\right]+\left[\begin{matrix}
		5&6\\
		1&3
		\end{matrix}\right]+\left[\begin{matrix}
		5&6\\
		1&4
		\end{matrix}\right]+\left[\begin{matrix}
		5&6\\
		2&5
		\end{matrix}\right]\\&+\left[\begin{matrix}
		5&6\\
		3&6
		\end{matrix}\right]+\left[\begin{matrix}
		5&6\\
		4&5
		\end{matrix}\right]+\left[\begin{matrix}
		5&6\\
		5&6
		\end{matrix}\right]+\left[\begin{matrix}
		5&6\\
		4&6
		\end{matrix}\right]
		\end{split}
		\]	
		
		\item[Edge 46] \[
		\begin{split}
		\hat{\tau}_{46}=&\left[\begin{matrix}
		4&6\\
		1&2
		\end{matrix}\right]+\left[\begin{matrix}
		4&6\\
		2&3
		\end{matrix}\right]+\left[\begin{matrix}
		4&6\\
		1&3
		\end{matrix}\right]+\left[\begin{matrix}
		4&6\\
		1&4
		\end{matrix}\right]+\left[\begin{matrix}
		4&6\\
		2&5
		\end{matrix}\right]\\&+\left[\begin{matrix}
		4&6\\
		3&6
		\end{matrix}\right]+\left[\begin{matrix}
		4&6\\
		4&5
		\end{matrix}\right]+\left[\begin{matrix}
		4&6\\
		5&6
		\end{matrix}\right]+\left[\begin{matrix}
		4&6\\
		4&6
		\end{matrix}\right]
		\end{split}
		\]		
\end{itemize}
\noindent After multiplicating these initiators, we obtain 
	\[
	\begin{split}
	Aut(\Gamma)=&\det(\mathcal{M}_\Gamma)\\
	=&\left(\begin{matrix}
	1&2&3&4&5&6\\
	1&2&3&4&5&6
	\end{matrix}\right)+\left(\begin{matrix}
	1&2&3&4&5&6\\
	2&1&3&5&4&6
	\end{matrix}\right)+\left(\begin{matrix}
	1&2&3&4&5&6\\
	3&2&1&6&5&4
	\end{matrix}\right)\\&+\left(\begin{matrix}
	1&2&3&4&5&6\\
	2&3&1&5&6&4
	\end{matrix}\right)+\left(\begin{matrix}
	1&2&3&4&5&6\\
	1&3&2&4&6&5
	\end{matrix}\right)+\left(\begin{matrix}
	1&2&3&4&5&6\\
	4&5&6&1&2&3
	\end{matrix}\right)\\&+\left(\begin{matrix}
	1&2&3&4&5&6\\
	6&5&4&3&2&1
	\end{matrix}\right)+\left(\begin{matrix}
	1&2&3&4&5&6\\
	5&6&4&2&3&1
	\end{matrix}\right)+\left(\begin{matrix}
	1&2&3&4&5&6\\
	6&4&5&3&1&2
	\end{matrix}\right)\\&+\left(\begin{matrix}
	1&2&3&4&5&6\\
	4&6&5&1&3&2
	\end{matrix}\right)+\left(\begin{matrix}
	1&2&3&4&5&6\\
	5&4&6&2&1&3
	\end{matrix}\right)+\left(\begin{matrix}
	1&2&3&4&5&6\\
	3&1&2&6&4&5
	\end{matrix}\right)
	\end{split}
	\]
\textbf{Calculations:} We shall do the calculations case by case. Each case corresponds to a term in \(\hat{\tau}_{12}\).
Then, we shall see 
\[
\begin{split}
\left[\begin{matrix}
1&2\\
1&2
\end{matrix}\right]\hat{\tau}_{23}&=\left[\begin{matrix}
1&2\\
1&2
\end{matrix}\right]\left[\begin{matrix}
2&3\\
2&3
\end{matrix}\right]+\left[\begin{matrix}
1&2\\
1&2
\end{matrix}\right]\left[\begin{matrix}
2&3\\
1&3
\end{matrix}\right]+\left[\begin{matrix}
1&2\\
1&2
\end{matrix}\right]\left[\begin{matrix}
2&3\\
1&4
\end{matrix}\right]+\left[\begin{matrix}
1&2\\
1&2
\end{matrix}\right]\left[\begin{matrix}
2&3\\
2&5
\end{matrix}\right]\\
&=\left(\begin{matrix}
1&2&3\\
1&2&3
\end{matrix}\right)+\left(\begin{matrix}
1&2&3\\
2&1&3
\end{matrix}\right)+\left(\begin{matrix}
1&2&3\\
2&1&4
\end{matrix}\right)+\left(\begin{matrix}
1&2&3\\
1&2&5
\end{matrix}\right)
\end{split}
\]
Now, since 
\[
\left[\begin{matrix}
1&2\\
1&2
\end{matrix}\right]\hat{\tau}_{13}=\left(\begin{matrix}
1&2&3\\
1&2&3
\end{matrix}\right)+\left(\begin{matrix}
1&2&3\\
1&2&4
\end{matrix}\right)+\left(\begin{matrix}
1&2&3\\
2&1&5
\end{matrix}\right)+\left(\begin{matrix}
1&2&3\\
2&1&3
\end{matrix}\right)
\]
we have 
\[
\left[\begin{matrix}
1&2\\
1&2
\end{matrix}\right]\hat{\tau}_{13}\hat{\tau}_{23}=\left(\begin{matrix}
1&2&3\\
1&2&3
\end{matrix}\right)+\left(\begin{matrix}
1&2&3\\
2&1&3
\end{matrix}\right)
\]
do this again, we have
\[
\left[\begin{matrix}
1&2\\
1&2
\end{matrix}\right]\hat{\tau}_{14}\hat{\tau}_{13}\hat{\tau}_{23}=\left(\begin{matrix}
1&2&3&4\\
1&2&3&4
\end{matrix}\right)+\left(\begin{matrix}
1&2&3&4\\
2&1&3&5
\end{matrix}\right)
\]
At the end of this process, we have 
\[
\left[\begin{matrix}
1&2\\
1&2
\end{matrix}\right]\hat{\tau}_{36}\hat{\tau}_{25}\hat{\tau}_{14}\hat{\tau}_{13}\hat{\tau}_{23}=\left(\begin{matrix}
1&2&3&4&5&6\\
1&2&3&4&5&6
\end{matrix}\right)+\left(\begin{matrix}
1&2&3&4&5&6\\
2&1&3&5&4&6
\end{matrix}\right)
\]
and it is not hard to check these are automorphisms;

\noindent Similarly, we have
\[
\left[\begin{matrix}
1&2\\
2&3
\end{matrix}\right]\hat{\tau}_{36}\hat{\tau}_{25}\hat{\tau}_{14}\hat{\tau}_{13}\hat{\tau}_{23}=\left(\begin{matrix}
1&2&3&4&5&6\\
3&2&1&6&5&4
\end{matrix}\right)+\left(\begin{matrix}
1&2&3&4&5&6\\
2&3&1&5&6&4
\end{matrix}\right);
\]
\[
\left[\begin{matrix}
1&2\\
1&3
\end{matrix}\right]\hat{\tau}_{36}\hat{\tau}_{25}\hat{\tau}_{14}\hat{\tau}_{13}\hat{\tau}_{23}=\left(\begin{matrix}
1&2&3&4&5&6\\
1&3&2&4&6&5
\end{matrix}\right)+\left(\begin{matrix}
1&2&3&4&5&6\\
3&1&2&6&4&5
\end{matrix}\right);
\]
\[
\left[\begin{matrix}
1&2\\
4&5
\end{matrix}\right]\hat{\tau}_{36}\hat{\tau}_{25}\hat{\tau}_{14}\hat{\tau}_{13}\hat{\tau}_{23}=\left(\begin{matrix}
1&2&3&4&5&6\\
5&4&6&2&1&3
\end{matrix}\right)+\left(\begin{matrix}
1&2&3&4&5&6\\
4&5&6&1&2&3
\end{matrix}\right);
\]
\[
\left[\begin{matrix}
1&2\\
5&6
\end{matrix}\right]\hat{\tau}_{36}\hat{\tau}_{25}\hat{\tau}_{14}\hat{\tau}_{13}\hat{\tau}_{23}=\left(\begin{matrix}
1&2&3&4&5&6\\
6&5&4&3&2&1
\end{matrix}\right)+\left(\begin{matrix}
1&2&3&4&5&6\\
5&6&4&2&3&1
\end{matrix}\right);
\]
\[
\left[\begin{matrix}
1&2\\
4&6
\end{matrix}\right]\hat{\tau}_{36}\hat{\tau}_{25}\hat{\tau}_{14}\hat{\tau}_{13}\hat{\tau}_{23}=\left(\begin{matrix}
1&2&3&4&5&6\\
6&4&5&3&1&2
\end{matrix}\right)+\left(\begin{matrix}
1&2&3&4&5&6\\
4&6&5&1&3&2
\end{matrix}\right);
\]
\[
\left[\begin{matrix}
1&2\\
1&4
\end{matrix}\right]\hat{\tau}_{13}\hat{\tau}_{23}=0;\;
\left[\begin{matrix}
1&2\\
2&5
\end{matrix}\right]\hat{\tau}_{13}\hat{\tau}_{23}=0;\;
\left[\begin{matrix}
1&2\\
3&6
\end{matrix}\right]\hat{\tau}_{13}\hat{\tau}_{23}=0.
\]
\end{exmp}

\bibliographystyle{plain}
\bibliography{tem}

\end{document}